\documentclass[a4paper,11pt]{amsart}

\usepackage{amsmath, amsthm, amsfonts, amssymb}
\usepackage[bookmarks=false]{hyperref}
\usepackage{enumerate}

\usepackage{color}

\numberwithin{equation}{section}

\newtheorem{letterthm}{Theorem}

\newtheorem{thm}{Theorem}[section]
\newtheorem{lem}[thm]{Lemma}

\newtheorem{prop}[thm]{Proposition}

\theoremstyle{definition}

\newtheorem*{exams}{Examples}

\newtheorem{df}[thm]{Definition}

\newtheorem*{definition}{Definition}

\newtheorem*{newclaim}{Claim}

\newcommand{\R}{\mathbf{R}}
\newcommand{\C}{\mathbf{C}}
\newcommand{\Z}{\mathbf{Z}}
\newcommand{\F}{\mathbf{F}}
\newcommand{\T}{\mathbf{T}}

\newcommand{\N}{\mathbf{N}}
\newcommand{\B}{\mathbf{B}}
\newcommand{\K}{\mathbf{K}}

\newcommand{\cI}{\mathcal{I}}

\newcommand{\id}{\text{\rm id}}

\newcommand{\rL}{\mathord{\text{\rm L}}}

\newcommand{\ran}{\mathord{\text{\rm ran}}}

\newcommand{\Vect}{\mathord{\text{\rm Vect}}}

\newcommand{\rd}{\: \mathrm{d}}

\begin{document}

\title[Evanescent isometric actions and WIE groups]{Evanescent affine isometric actions and weak identity excluding groups}

\address{CNRS, ENS Lyon, Unit\'e de Math\'ematiques Pures et Appliqu\'ees, 46, all\'ee d'Italie, 69364 Lyon}

\author{Amine Marrakchi}
\email{amine.marrakchi@ens-lyon.fr}

%\thanks{The author was a JSPS International Research Fellow (PE18760)}



\begin{abstract}
We investiguate a property of affine isometric actions on Hilbert spaces called evanescence. Evanescent actions are the extreme opposite of irreducible actions. Every affine isometric action decomposes naturally into an evanescent part and an irreducible part. We study when this decomposition is unique. We also study when an action that has almost fixed points is automatically evanescent. We relate these questions to the identity excluding property for groups. We also relate them to the finiteness of the von Neumann algebras generated by the linear part of the action.
\end{abstract}

\maketitle

%%%%%%%%%%

%\section{Introduction}

%\tableofcontents

%\section{Preliminaries}
%\subsection*{Ultraproduct von Neumann algebras}
%Let $M$ be any von Neumann algebra. Let $I$ be any nonempty directed set and $\omega$ any cofinal ultrafilter on $I$. Define
%\begin{align*}
%	\ell^\infty(I,M) &= \{ (x_i)_{i\in I} \mid x_i\in M, \ \sup_i \|x_i\|_\infty <\infty \} \\
%	\mathcal I_{\omega} &= \left\{ (x_i)_{i} \in \ell^\infty(I, M) \mid x_i \to 0 \text{ *-strongly as } i \to \omega \right\} \\
%	\mathcal M^{\omega} &= \left \{ x \in \ell^\infty(I, M) \mid  x \mathcal I_{\omega} \subset \mathcal I_{\omega} \text{ and } \mathcal I_{\omega}x \subset \mathcal I_{\omega}\right\}.
%\end{align*}
%The quotient C$^*$-algebra $M^\omega := \mathcal{M}^\omega/ \mathcal{I}_\omega$ is in fact a von Neumann algebra, and we call it the \textit{ultraproduct von Neumann algebra} \cite{Oc85}. For more on ultraproduct von Neumann algebras, we refer the reader to \cite{Oc85,AH12}.

%\appendix

\section{Introduction}
The study of affine isometric actions of locally compact groups on Hilbert spaces is nowadays an important topic in representation theory and geometric group theory (see \cite{BHV08} and the references therein). In \cite{BPV16}, \emph{irreducible} affine isometric actions were systematically studied. These are the actions that do not admit any non-trivial invariant affine subspace. In this paper, we study \emph{evanescent} affine isometric actions. These are the actions that admit \emph{arbitrarily small} invariant affine subspaces. Thus, evanescent actions are the complete opposite of irreducible ones. We refer to Section \ref{evanescent section} for the precise definitions.

The property of evanescence was introduced in \cite{AIM19}, with motivations coming from ergodic theory (see also \cite[Proposition 1.7]{DZ21} for a very recent application). In \cite{AIM19}, it was observed that every affine isometric action decomposes naturally into an irreducible part and an evanescent part and that every evanescent action has \emph{almost fixed points}. With these two facts in mind, it is natural to ask the following questions :
\begin{itemize}
\item[(Q1)] Is an action with almost fixed points necessarily evanescent?
\item[(Q2)] Is the decomposition of a given action into an irreducible part and an evanescent part unique?
\end{itemize}
The goal of this paper is to study these two questions. Our first main result identifies a class of groups for which we have a very satisfactory answer for both (Q1) and (Q2). This is class of Weak Identity Excluding groups.

\begin{definition}
We say that a locally compact group $G$ is \emph{Weak Identity Excluding} (WIE) if every irreducible representation of $G$ that has almost invariant vectors is trivial.
\end{definition}
This property is a very natural weakening of the \emph{Identity Excluding} property introduced in \cite{JRW96} (see also \cite{LW95}). For discrete groups, the two notions coincide.

\begin{exams}
The following groups are WIE :
\begin{itemize}
\item CCR groups (See \cite{BH19}), such as connected semisimple Lie groups or reductive algebraic groups over local fields.
\item Nilpotent groups (see \cite{LW95}).
\item Property (T) groups (by definition).
\end{itemize}
An example of a group that is not WIE is the free group $\F_n, \: n \geq 2$. It is not known whether every \emph{finitely generated} WIE group is either virtually nilpotent or has property (T) (see also \cite{Ja04}).
\end{exams}

Here is our main result.

\begin{letterthm} \label{equivalence WIE}
Let $G$ be a locally compact group. Then the following are equivalent :
\begin{enumerate}[\rm (i)]
\item $G$ is WIE.
\item Every affine isometric action of $G$ that has almost fixed points is evanescent.
\item Every affine isometric action of $G$ has a unique decomposition into an irreducible part and an evanescent part.
\end{enumerate}
\end{letterthm}

The proof of the implication $(\rm iii) \Rightarrow (\rm i)$ in Theorem \ref{equivalence WIE} is deduced from the following more precise theorem which exhibits some wild behaviour of affine isometric actions of groups that are not WIE.

\begin{letterthm} \label{not unique decomposition}
Let $G$ be a locally compact group. Suppose that $G$ is not WIE and let $\pi$ be a non-trivial irreducible representation of $G$ that has almost invariant vectors. Then there exists an affine isometric action $\alpha$ of $G$ such that :
\begin{enumerate}[\rm (i)]
\item The linear part of $\alpha$ is a multiple of $\pi$.
\item $\alpha$ is evanescent.
\item $\alpha$ admits a decomposition into a sum of an evanescent action and a nontrivial irreducible action.
\item $\alpha$ admits a decomposition into a sum of two irreducible  actions.
\end{enumerate}
\end{letterthm}

Finally, we can also answer (Q1) and (Q2) by adding an assumption on the von Neumann algebra generated by the linear part of the isometric aciton, and without any assumption on the group.
\begin{letterthm} \label{main finite}
Let $G$ be a locally compact group. Let $\alpha$ be an affine isometric action with linear part $\pi$.
\begin{enumerate}[\rm (i)]
\item If $\pi(G)''$ is a finite von Neumann algebra then $\alpha$ is evanescent if and only if it has almost fixed points.
\item If $\pi(G)'$ or $\pi(G)''$ is a finite von Neumann algebra, then $\alpha$ has a unique decomposition into an irreducible part and an evanescent part.
\end{enumerate}
\end{letterthm}

It is known that every affine isometric action of an amenable group whose linear part is contained in a multiple of the left regular representation has almost fixed points. Recently, Bekka and Valette proved that such an action is in fact evanescent. We note that $\pi(G)''$ is always finite when $G$ is discrete and $\pi$ is contained in a multiple of the left regular representation of $G$. Thus, Theorem \ref{main finite}.(i) is a generalization of the result of Bekka and Valette.

\subsection*{Conventions}
In this paper, all Hilbert spaces are assumed to be \emph{separable} and all locally compact groups are assumed to be \emph{second countable.} 

It is often more natural to look at affine isometric actions of groups on real Hilbert spaces. For this reason, all results in this paper will deal with both the real case and the complex case. In particular, the main results of this paper are valid both in the real and the complex context.

\section{Preliminaries on representation theory}

\subsection{Unitary and orthogonal representations}
A \emph{unitary representation} of a locally compact group $G$ is a continuous homomorphism $\pi : G \rightarrow U(H)$ where $U(H)$ is the unitary group of some complex Hilbert space $H$, equipped with the topology of pointwise norm convergence.

An \emph{orthogonal representation} of a locally compact group $G$ is a continuous homomorphism $\pi : G \rightarrow O(H)$ where $O(H)$ is the orthogonal group of some real Hilbert space $H$, equipped with the topology of pointwise norm convergence.

When we simply use the word \emph{representation} without specification, it will mean that the representation can be both unitary or orthogonal. Similarly, a \emph{Hilbert space} can be both a complex Hilbert space or a real Hilbert space.

If $\rho : G \curvearrowright H$ is an orthogonal representation, we denote  by $\rho_\C : G \curvearrowright H_\C$ the associated unitary representation where $H_\C=H \otimes_\R \C$ is the complexification of $H$. Observe that $(\rho_\C)_\R \cong \rho \oplus \rho$.

 If $\pi : G \curvearrowright H$ is a unitary representation, we denote by $\pi_\R : G \curvearrowright H_\R$ the orthogonal asociated representation where $H_\R$ is the real Hilbert space obtained from $H$ by restricting the scalars. We denote by $\pi^*$ the conjugate representation of $\pi$ on $\overline{H}$ where $\overline{H}$ is the conjugate Hilbert space of $H$, i.e.\ we have a conjugate linear isometry $\xi \mapsto \bar{\xi}$ from $H$ to $\overline{H}$. Observe that $H \oplus \overline{H}$ is naturally isomorphic to $H_\R \otimes_\R \C$ as a complex Hilbert space via the map
$$ (\xi, \bar{\eta}) \mapsto (\xi + \eta) \otimes_\R 1 + (i\xi-i\eta) \otimes_\R i$$
hence $(\pi_\R)_\C \cong \pi \oplus \pi^*$.

\begin{prop} \label{irreducible orthogonal}
Let $G$ be a locally compact group. Let $\rho : G \curvearrowright H$ be an irreducible orthogonal representation of $G$. The following are equivalent :
\begin{enumerate}[\rm (i)]
\item $\rho(G)' \neq \R$.
\item $\rho_\C$ is not irreducible.
\item $\rho=\pi_\R$ for some irreducible unitary representaiton $\pi$.
\end{enumerate}
\end{prop}
\begin{proof}
$(\rm i) \Rightarrow (\rm ii)$ Suppose that $\rho(G)' \neq \R$. Then $\rho_\C(G)'=\rho(G)' \otimes_\R \C \neq \C$, hence $\rho_\C$ is not irreducible by Schur's lemma.

$(\rm ii) \Rightarrow (\rm iii)$ Suppose that $\rho_\C$ is not irreducible. Write $\rho_\C=\pi \oplus \sigma$ where $\pi$ and $\sigma$ are two non-zero unitary representations of $G$. Then we have
$$ \rho \oplus \rho = (\rho_\C)_\R = \pi_\R \oplus \sigma_\R$$
as orthogonal representations. Since $\rho$ is irreducible, this forces $\rho \cong \pi_\R \cong \sigma_\R$.

$(\rm iii) \Rightarrow (\rm i)$ is obvious.
\end{proof}

\begin{prop} \label{irreducible unitary}
Let $G$ be a locally compact group. Let $\pi : G \curvearrowright H$ be an irreducible unitary representation. The following are equivalent :
\begin{enumerate}[\rm (i)]
\item $\pi \cong \pi^*$.
\item $\pi_\R$ is not irreducible.
\item $\pi=\rho_\C$ for some irreducible orthogonal representaiton $\rho$.
\end{enumerate}
\end{prop}
\begin{proof}
$(\rm i) \Rightarrow (\rm ii)$ Suppose that $\pi \cong \pi^*$. In other words, there exists a conjugate-linear isometry $J$ of $H$ that commutes with $\pi$. Then $J^2$ is a $\C$-linear isometry that commutes with $\pi$. Since $\pi$ is irreducible, Schur's lemma implies that $J^2=\lambda$ for some $\lambda \in \C$, with $|\lambda|=1$. Then we obtain a non-trivial decomposition $H_\R=\ker(J-\mu) \oplus \ker(J+\mu)$ where $\mu^2=\lambda$ and this decomposition is invariant under $\pi_\R$, hence $\pi_\R$ is not irreducible.

$(\rm ii) \Rightarrow (\rm iii)$ Suppose that $\pi_\R$ is not irreducible. Let $K \subset H$ be a non-trivial invariant $\R$-linear subspace. Then $K+iK$ is dense in $H$ and $K \cap iK =\{0\}$. Define a closed densely defined $\R$-linear operator $J$ on $H$ by $J(\xi+i\eta)=\xi-i\eta$ for all $\xi, \eta \in K$. Observe that $J$ is conjugate-linear on $H$, thus $J^*J$ is $\C$-linear. Since $J^*J$ commutes with $\pi$ and $\pi$ is irreducible, then $J^*J=\lambda $ for some $\lambda > 0$. Since $J^2=1$,we have $\lambda=1$. Thus $J$ is an orthogonal operator and $K=\ker(J-1)$ is orthogonal to $iK=\ker(J+1)$. Let $\rho$ be the restriction of $\pi$ to $K$. We proved that $\pi=\rho_\C$. 

$(\rm iii) \Rightarrow (\rm i)$ is obvious.
\end{proof}

\subsection{Weak Identity excluding groups}
We say that a locally compact group $G$ is WIE (for Weak Identity Excluding) if it satisfies the following equivalent conditions.

\begin{prop}
Let $G$ be a locally compact group. Then the following are equivalent :
\begin{enumerate}[ \rm (i)]
\item Every irreducible orthogonal representation of $G$ that has almost invariant vectors is trivial.
\item Every irreducible unitary representation of $G$ that has almost invariant vectors is trivial.
\end{enumerate}
\end{prop}
\begin{proof}
A unitary representation $\pi$ has almost invariant vectors if and only if $\pi_\R$ has almost invariant vectors. An orthogonal representation $\rho$ has almost invariant vectors if and only if $\rho_\C$ has almost invariant vectors. Hence the Proposition follows from Proposition \ref{irreducible orthogonal} and \ref{irreducible unitary}.
\end{proof}

\subsection{Cohomology}
Let $\pi$ be a (real or complex) representation of a locally compact group $G$ on a Hilbert space $H$. We denote by $Z^1(G,\pi)$ the set of all continuous $1$-cocycles for the representation $\pi$, i.e.\ the set of all continuous maps $c : G \rightarrow H$ such that
$$ c(gh)=c(g)+\pi(g)c(h)$$
for all $g,h \in G$. We put on $Z^1(G,\pi)$ the topology of uniform convergence on compact subsets. We denote by $B^1(G,\pi) \subset Z^1(G,\pi)$ the subset of all coboundaries, i.e.\ cocycles of the form
$ c(g)=\xi-\pi(g)\xi$ for some $\xi \in H$. The closure of $B^1(G,\pi)$ inside $Z^1(G,\pi)$ is denoted $\overline{B^1(G,\pi)}$. Its elements are called \emph{almost coboundaries}.

Observe that $Z^1(G,\pi)$ is a module over $\pi(G)'$. This module structure has been observed in \cite{BPV16} as well as \cite{BV97}. It has been exploited fruitfully in \cite{Be17}.

\subsection{Affine isometric actions}
Let $H$ be a (real or complex) Hilbert space. We denote by $\mathrm{Aff}(H)$ the topological group of all affine isometries of $H$, equipped with the topology of pointwise norm convergence on $H$. An \emph{affine isometric action} of a locally compact group $G$ on $H$ is a continuous group homomorphism $\alpha : G \rightarrow \mathrm{Aff}(H)$.  Then $\alpha$ is of the form
$$ \alpha_g(\xi)=\pi(g) \xi + c(g), \quad g \in G, \xi \in H$$
for some representation $\pi$ of $G$ on $H$ and some cocycle $c \in Z^1(G,\pi)$.

We note that $c \in B^1(G,\pi)$ if and only if $\alpha$ has a fixed point and $c \in \overline{B^1(G,\pi)}$ if and only if $\alpha$ has \emph{almost fixed points}, i.e.\ for every compact subset $K \subset G$ and every $\varepsilon > 0$ there exists $\xi \in H$ such that 
$$ \forall g \in K, \quad \| \alpha_g(\xi)-\xi \| \leq \varepsilon.$$

\subsection{Sums of affine isometric actions} 
Let $\alpha_i : G \curvearrowright H_i, \; i=1,2$ be two affine isometric actions with linear parts $\pi_i : G \curvearrowright H_i$ and cocycles $c_i \in Z^1(G,\pi_i)$. Then 
$$c_1 \oplus c_2 : G \ni g \mapsto (c_1(g),c_2(g)) \in H_1 \oplus H_2$$
is a cocycle for $\pi_1 \oplus \pi_2$. Thus, we can define a new affine isometric action $\alpha_1 \oplus \alpha_2: G \curvearrowright H_1 \oplus H_2$ associated to the cocycle $c_1 \oplus c_2  \in Z^1(G, \pi_1 \oplus \pi_2)$.

\subsection{Projected actions} 
Let $\alpha : G \curvearrowright H$ be an affine isometric action with linear part $\pi$ and cocycle $c \in Z^1(G,\pi)$. Let $p \in \pi(G)'$ be a projection. Let $p(\pi)$ be the representation of $G$ obtained by restricting $\pi$ to $p(H)$. Then $pc$ is a cocycle for $p(\pi)$. Thus, we can define the \emph{projected} action $p(\alpha) : G \curvearrowright p(H)$ as the affine isometric action associated to this cocycle $pc \in Z^1(G,p(\pi))$. Observe that $\alpha = p(\alpha) \oplus p^{\perp}(\alpha)$.

\subsection{The laplacian operator}
\begin{df} Let $\mu$ be a probability measure on a locally compact group $G$. We will say that $\mu$ is \emph{adapted} if :
\begin{enumerate}[ \rm (i)]
\item $\mu$ is symmetric.
\item $\mu$ is absolutely continuous with respect to the Haar measure.
\item the support of $\mu$ generates $G$.
\end{enumerate}
Suppose that $G$ is compactly generated. We say that $\mu$ has a \emph{second moment} if $$ \int_G |g|^2 \rd \mu(g) < +\infty $$
where $| \cdot |$ is the word-length function associated to any compact generating set.
\end{df}

\begin{prop}
Let $\pi : G \curvearrowright H$ be a representation that has no invariant vectors. Let $\mu$ be an adapted probability measure on $G$. Let $\Delta_\mu \in \pi(G)''$ be the laplacian
$$ \Delta_\mu= \int_G |1-\pi(g)|^2 \rd \mu(g).$$
Then $\Delta_\mu$ is injective and $\pi$ has almost invariant vectors if and only if $0$ is in the spectrum of $\Delta_\mu$.
\end{prop}
\begin{proof}
See \cite[Proposition G.4.2]{BHV08}.
\end{proof}

\subsection{$\rL^2$-integrable cocycles}
Let $\pi : G \curvearrowright H$ be a representation of a locally compact group $G$. Let $\mu$ be an adapted probability measure on $G$. We define 
$$Z^1_\mu(G,\pi)=\{ c \in Z^1(G,\pi) \mid \int_G \| c(g)\|^2 \rd \mu(g) < +\infty \}.$$

\begin{prop}
The space $Z^1_\mu(G,\pi)$ is a closed subspace of $\rL^2(G,\mu,H)$.
The inclusion map $\iota  : Z^1_\mu(G,\pi) \rightarrow Z^1(G,\pi)$ is continuous, i.e.\ for every compact subset $Q \subset G$, there exists $\kappa > 0$ such that
$$ \forall c \in  Z^1_\mu(G,\pi) , \quad  \sup_{g \in Q} \| c(g) \| \leq \kappa \left( \int_G \| c(g)\|^2 \rd \mu(g) \right)^{1/2}.$$
Moreover, if $G$ is compactly generated and $\mu$ has a second moment, then $\iota$ is a (surjective) homeomorphism.
\end{prop}
\begin{proof}
See the proof of \cite[Lemma 2.1]{EO18}.
\end{proof}

We view $Z^1_\mu(G,\pi)$ as a Hilbert space since it is a closed subspace of $\rL^2(G,\mu,H)$. We let $\overline{B}^1_\mu(G,\pi)$ denote the closure of $B^1(G,\pi)$ inside $Z^1_\mu(G,\pi)$. Note that $\overline{B}^1_\mu(G,\pi)$ is in general smaller than $\overline{B^1(G,\pi)} \cap Z_\mu^1(G,\pi)$, because the topology of $Z^1_\mu(G,\pi)$ is stronger than the topology of $Z^1(G,\pi)$. 

We will need the following easy lemma.

\begin{lem}
Let $\pi : G \curvearrowright H$ be a representation and take $c \in \overline{B^1(G,\pi)}$. Then there exists an adapted probability measure $\mu$ on $G$ such that $c \in \overline{B}_\mu^1(G,\pi)$.
\end{lem}

%\begin{thm}
%Let $\pi : G \curvearrowright H$ be a representation that has no invariant vectors. Let $\mu$ be an adapted probability measure on $G$. There exists a continuous injective linear map $U : H \rightarrow \overline{B^1(G,\pi)}$ such that 
%\begin{enumerate}[ \rm (i)]
%\item $U$ is $\pi(G)'$-equivariant.
%\item For all $\xi \in H$, $U\xi \in B^1(G,\pi)$ if and only if $\xi$ is in the range of $\Delta_\mu^{1/2}$.
%\item If $G$ is compactly generated and $\mu$ has a second moment, then $U$ is a surjective isomorphism.
%\end{enumerate}
%\end{thm}

\begin{prop} \label{polar decomposition}
Let $\pi : G \curvearrowright H$ be a representation that has no invariant vectors. Consider the coboundary operator 
$$ \partial_\mu  : H \ni \xi \mapsto ( g \mapsto \xi -\pi(g)\xi) \in Z_\mu^1(G,\pi).$$
Then
$$ \partial_\mu^*\partial_\mu=\Delta_\mu= \int_G |1-\pi(g)|^2 \rd \mu(g)$$
Let $\partial_\mu=U_\mu\Delta_\mu^{1/2}$ be the polar decomposition of $\partial_\mu$. Then $U_\mu$ is an isometry from $H$ onto $\overline{B}^1_\mu(G,\pi)$ and for $\xi \in H$, we have $U_\mu \xi \in B^1(G,\pi)$ if and only if $\xi$ is in the range of $\Delta_\mu^{1/2}$. 
\end{prop}
\begin{proof}
We can compute its adjoint
$$ \partial_\mu^* : Z_\mu^1(G,\pi) \ni c \mapsto 2\int_G c(g) \rd \mu(g) \in H$$
hence 
$$\partial_\mu^* \partial_\mu :  \xi \mapsto 2\int_G (\xi-\pi(g)\xi) \rd \mu(g) $$
and, using the fact that $\mu$ is symmetric, we get $ \partial_\mu^*\partial_\mu=\Delta_\mu.$
Since $\partial_\mu$ is injective by assumption on $\pi$ and its range is $B^1(G,\pi)$ by definition, we get that $U_\mu$ is an isometry from $H$ onto $\overline{B}^1_\mu(G,\pi)$. Finally, it is clear that $U_\mu \xi$ is in the range of $\partial_\mu$ if and only if $\xi$ is in the range of $\Delta_\mu^{1/2}$.
\end{proof}

\section{Preliminaries on von Neumann algebras}
Let $H$ be Hilbert space. For every subset $X \subset \B(H)$, we denote by $X'=\{ S \in \B(H) \mid \forall T \in X, \; ST=TS \}$ the commutant of $X$ and by $X''=(X')'$ is the bicommutant of $X$.

We recall that the weak topology on $\B(H)$ is the finest topology that makes all the linear maps $T \mapsto \langle T \xi, \eta \rangle$ for $\xi,\eta \in H$ continuous.

\begin{thm}[von Neumann bicommutant's theorem]
Let $H$ be a (real or complex) space. Let $M \subset \B(H)$ a unital subalgebra such that $T^* \in M$ for all $T \in M$. Then the bicommutant $M''$ is the closure of $M$ for the weak topology.
\end{thm}

Recall that a (real or complex) von Neumann algebra $M$ is a weakly closed self-adjoint unital subalgebra of $\B(H)$ where $H$ is a (real or complex) Hilbert space.

We say that a von Neumann algebra is \emph{finite} if it admits a faithful normal tracial state, i.e.\ a weakly continuous linear functional $\tau : M \rightarrow \K$ (where $\K=\R$ or $\C$) such that $\tau(1)=1$ and $\tau(x^*x) \geq 0$ for all $x \in M$ with equality if and only if $x=0$. 

If $M$ is a finite von Neumann algebra and $\tau$ is a faithful normal tracial state on $M$, we define the Hilbert space $\rL^2(M,\tau)$ as the completion of $M$ with respect to the $2$-norm $\|x\|_2=\tau(x^*x)^{1/2}$. We also define the topology of convergence in measure on $M$ by the following distance
$$ d_\tau(x,y)=\tau\left( \min(1,|x-y|) \right).$$
This topology does not depend on the choice of $\tau$ and turns $M$ into a topological algebra. The completion of $M$ with respect to this topology is a complete topolgical algebra that we denote by $\mathcal{A}(M)$. We can identify $\rL^2(M,\tau)$ with a subspace of $\mathcal{A}(M)$, and $M$ acts by left and right multiplication on $\rL^2(M,\tau)$.

\subsection{Ideals in von Neumann algebras} 
It is well-known that every weakly closed left ideal in a von Neumann algebra $M$ is of the form $Mp$ for some projection $p \in M$.
\begin{lem} \label{dense ideal increasing}
Let $M$ be a von Neumann algebra and $I \subset M$ a left ideal. Let $Mp$ be the weak closure of $I$ where $p \in M$ is some projection. Then there exists an increasing sequence of projections $p_n \in I$ such that $\lim_n p_n =p$.
\end{lem}
\begin{proof}
Since $I$ is a weakly dense left ideal in $Mp$, then $I^* \cap I$ is a weakly dense subalgebra of $pMp$. Thus, by Kaplansky's density theorem, we can find a sequence of positive contractions $x_n \in  I$ that converges to $p$ in the strong topology. Let $q_n=1_{[1/2,1]}(x_n)$, then $q_n$ is a sequence of projection that converges to $p$.

Now, we want to construct inductively an increasing sequence $p_n \in I$ such that $p_n$ converges to $p$. For this, it is enough to show that for every projection $p_1 \in I$, we can find $p_2 \in I$ with $p_2 \geq p_1$ and $p_2$ arbitrarily close to $p$. So take $p_1 \in I$.  Let $$y_n=1-(1-p_1)(1-q_n)(1-p_1)=p_1q_np_1+p_1+q_n-p_1q_n-q_np_1 \in I.$$ We can check easily that $y_n$ commutes with $p_1$ and $p_1 \leq y_n \leq p$. Moreover, when $n \to \infty$, we have $y_n \to p$. Thus we can take $p_2=1_{[1/2,1]}(y_n)$ for $n$ large enough ($p_2$ is larger than $p_1$ because $y_n$ commutes with $p_1$ and $y_n \geq p_1$).
\end{proof}

\begin{lem} \label{dense ideal intersection}
Let $M$ be a finite von Neumann algebra. Let $I,J \subset M$ be two dense left ideals. Then $I \cap J$ is also dense.
\end{lem}
\begin{proof}
Let $\tau$ a faithful normal tracial state on $M$. Then we have 
$$ \tau(p)+\tau(q)=\tau(p \vee q)+\tau(p \wedge q)$$
for all projections $p,q \in M$. Let $p_n \in I$ and $q_n \in I$ be two sequences of projections such that $\lim_n p_n=\lim_n q_n=1$. Then the equation above shows that $\lim_n \tau(p_n \wedge q_n)=1$, hence $\lim_n p_n \wedge q_n=1$. Since $p_n \wedge q_n \in I \cap J$ for all $n$, we conclude that $I \cap J$ is dense.
\end{proof}

\begin{lem} \label{dense ideal finite}
Let $M$ be a finite von Neumann algebra. Let $I \subset M$ be a dense left ideal. Then for every $a \in M$, the left ideal $J=\{ x \in M \mid xa \in I\}$ is also dense.
\end{lem}
\begin{proof}
We may assume without loss of generality that $a$ is positive. Indeed, in the general case, if $a=v|a|$ is the polar decomposition of $a$ and $p_n \in M$ is a sequence of projections converging to $1$ such that $p_n|a| \in I$, then the sequence $x_n =v p_n v^*+(1-vv^*)$ is in $J$ and converges to $1$.

Now, we can assume that $a$ is positive. Let $b$ be a pseudo-inverse of $a$ in $\mathcal{A}(M)$, i.e. $ab=ba=s$ where $s$ is the support of $a$. Let $q \in M$ be a projection with $q \leq s$ such that $q b$ is bounded. Let $p_n \in I$ such that $\lim_n p_n=1$. Let $r_n= q \wedge p_n$. We claim that $\lim_n r_n = q$. Indeed, since $M$ is finite, it admits a faithful normal tracial state $\tau$. Then we have
$$ \tau(q)+\tau(p_n)=\tau(r_n)+\tau(q \vee p_n).$$
Since $\lim_n \tau(p_n)=1$ and $\tau(q \vee p_n) \leq 1$, we conclude that $\lim_n \tau(r_n)=\tau(q)$, i.e. $\lim_n r_n=q$ as we claimed. Since $r_n \in I$, we then have $r_nb \in J$ for all $n \in \N$. Hence $qb$ is in the closure $\overline{J}$ of $J$. Thus, $aqb=qab=q$ is also in $\overline{J}$. Since this holds for every $q \leq s$ such that $qb$ is bounded, we conclude that $s \in \overline{J}$. Obviously, we have $s^{\perp} \in J$. Thus $1=s+s^{\perp} \in \overline{J}$ and we are done.
\end{proof}

\begin{lem} \label{approximation finite}
Let $M \subset \B(H)$ be a (real or complex) finite von Neumann algebra. Let $T \in M$ be an injective operator. Then for every $\xi \in H$, the left ideal $\{ S \in M' \mid S \xi \in \ran(T)\}$ is dense in $M'$.
\end{lem}
\begin{proof}
We will prove that there exists an increasing sequence of projections $P_n \in M'$ with $\lim_n P_n=\id$ such that $P_n(\xi) \in \ran(T)$ for all $n$.

Fix a faithful finite trace $\tau$ on $M$. We first prove this when $H=\rL^2(M,\tau)q$ for some projection $q \in M$ and $M$ acts by left multiplication. We view $\rL^2(M,\tau)$ as a subspace of $\mathcal{A}(M)$. Since it is an injective operator, we know that $T$ is invertible in $\mathcal{A}(M)$. Take $\xi \in \rL^2(M,\tau)$. Consider $T^{-1}\xi \in \mathcal{A}(M)$. Then we can find an increasing sequence of projections $p_n \in M$ with $\lim p_n=q$ such that $T^{-1}\xi p_n \in M$ for all $n$. A fortiori $T^{-1}\xi p_n \in \rL^2(M,\tau)q$. This implies that $\xi p_n \in T \rL^2(M,\tau)q$. Now, let $P_n \in M'$ be the projection obtained by right multiplication by $p_n$ on $H=\rL^2(M,\tau)q$. Then by construction, we have $\lim_n P_n=\id_H$ and $P_n(\xi) \in \ran(T)$ for all $n$.

In the general case, we use the fact that every Hilbert $M$-module is of the form $H=\bigoplus_{n \in \N} \rL^2(M,\tau)q_n$ for some sequence of projections $q_n \in M$ where $M$ acts by left multiplication. Then if $\xi \in H$, we can write $\xi =\oplus_n \xi_n$ for some sequence $\xi_n \in \rL^2(M,\tau)q_n$. For each $\xi_n$, we can find a sequence $(p_{n,m})_{m \in \N}$ of projections in $M$ that increases to $q_n$ and such that $\xi_n p_{n,m}$ is in the range of $T$ for all $n,m \in \N$. We then define $P_n \in \B(H)$ to be the operator of right multiplication by
$$ \bigoplus_{m \leq n} p_{n,m}.$$
The sequence of operators $P_n \in M'$ satisfies the desired conclusion.
\end{proof}

\section{Evanescent affine isometric actions} \label{evanescent section}
Let $\pi$ be a representation of a locally compact group $G$ on a Hilbert space $H$. Then $Z^1(G,\pi)$ is a module over the von Neumann algebra $\pi(G)'$. For every $c \in Z^1(G,\pi)$, we define the \emph{reduction ideal} of $c$ by
$$ \mathcal{I}(c)=\{ T \in \pi(G)' \mid Tc \in B^1(G,\pi) \}.$$
It is indeed a left ideal of $\pi(G)'$ which is not necessarily weakly closed. Note that $\mathcal{I}(c)$ depends only on the cohomology class of $c$.
\begin{df}
We say that the cocycle $c$ is :
\begin{itemize}
\item \emph{Irreducible} if $\mathcal{I}(c)=\{0\}$.
\item \emph{Evanescent} if $\mathcal{I}(c)$ is dense in $\pi(G)'$.
\end{itemize}
\end{df}
We say that an affine isometric action $\alpha$ is irreducible (resp.\ evanescent) if the corresponding cocycle is irreducible (resp.\ evanescent).

We denote by $\mathcal{E}(G,\pi)$ the set of all evanescent cocycles for the representation $\pi$. Note that $\mathcal{E}(G,\pi)$ is a priori not a vector space.

Using Lemma \ref{dense ideal increasing}, one sees that the definition of evanescent affine isometric actions given above corresponds to the one given in \cite{AIM19}.

The following properties (already observed in \cite{AIM19}) follow directly from the definitions.

\begin{prop}
Let $\pi$ be a representation of a locally compact group $G$ on a Hilbert space $H$ and take $c \in Z^1(G,\pi)$. If $c$ is evanescent then $c \in \overline{B^1(G,\pi)}$.
\end{prop}

\begin{prop} \label{natural decomposition}
Let $\alpha : G \curvearrowright H$ be an affine isometric action with linear part $\pi$ and cocycle $c \in Z^1(G,\pi)$. Let $p \in \pi(G)'$ be the projection such that $\overline{\cI(c)}=\pi(G)'p$. Then $p(\alpha)$ is evanescent and $p^\perp(\alpha)$ is irreducible.

Moreover, for every projection $q \in \pi(G)'$ such that $q(\alpha)$ is evanescent, we have $q \leq p$.
\end{prop}
Thus, every affine isometric action has a natural direct sum decomposition into an irreducible part and an evanescent part. However, as we will see from Theorem \ref{not unique decomposition}, this decomposition is \emph{not} unique in general.

\section{Proofs of the main theorems}
\begin{prop} \label{module unique decomposition}
Let $G$ be a locally compact group and let $\pi$ be a representation such that $\mathcal{E}(G,\pi)$ is a $\pi(G)'$-module. Then every affine isometric action with linear part $\pi$ admits a unique decomposition into a direct sum of an irreducible action and an evanescent action.
\end{prop}
\begin{proof}
Let $\alpha$ be an affine isometric action with linear part $\pi$. Let $c \in Z^1(G,\pi)$ be the corresponding cocycle. Let $p \in \pi(G)'$ be the projection such that $\overline{\cI(c)}=\pi(G)'p$ as in Proposition \ref{natural decomposition}. Let $q \in \pi(G)'$ be a projection such that $q(\alpha)$ is evanescent and $q^{\perp}(\alpha)$ is irreducible. Then $q \leq p$ by Proposition \ref{natural decomposition}. Let $r=q^{\perp}p=pq^{\perp}$. Since $pc$ is evanescent and $\mathcal{E}(G,\pi)$ is a $\pi(G)'$-module, we know that $rc$ is also evanescent, hence $r(\alpha)$ is evanescent. But $q^{\perp}(\alpha)$ is irreducible and $r \leq q^{\perp}$, hence 
$r(\alpha)$ is also irreducible. Therefore, we must have $r=0$, i.e.\ $p=q$.
\end{proof}

\begin{prop} \label{finite implies evanescent}
Let $G$ be a locally compact group and let $\pi$ be a (orthogonal or unitary) representation such that $\pi(G)''$ is a finite (real or complex) von Neumann algebra. Then $\overline{B^1(G,\pi)}=\mathcal{E}(G,\pi)$.
\end{prop}
\begin{proof}
Take $c \in \overline{B^1(G,\pi)}$. We can find an adapted probability measure $\mu$ on $G$ such that $c=U_\mu \xi$ for some $\xi \in H$. By Lemma \ref{approximation finite}, we can find an increasing sequence $p_n \in \pi(G)'$ with $\lim_n p_n=1$ such that $p_n\xi$ is in the range of $\Delta_\mu^{1/2}$ for all $n \in \N$. This means that $p_n c \in B^1(G,\pi)$ for all $n \in \N$ and we conclude that $c$ is evanescent.
\end{proof}

\begin{prop} \label{finite implies module}
Let $G$ be a locally compact group and let $\pi$ be a representation such that $\pi(G)'$ is a finite von Neumann algebra. Then $\mathcal{E}(G,\pi)$ is a $\pi(G)'$-module.
\end{prop}
\begin{proof}
Take $c_1,c_2 \in \mathcal{E}(G,\pi)$. Then $\mathcal{I}(c_1) \cap \mathcal{I}(c_2) \subset \mathcal{I}(c_1+c_2)$. By Lemma \ref{dense ideal intersection}, this implies that $c_1+c_2$ is evanescent. This shows that $\mathcal{E}(G,\pi)$ is a vector space. Now, take $c \in \mathcal{E}(G,\pi)$ and $a \in \pi(G)'$. Then $\{ x \in \pi(G)' \mid xa \in \mathcal{I}(c) \} \subset \cI(ac)$ and Lemma \ref{dense ideal finite} shows that $a c$ is still in $\mathcal{E}(G,\pi)$. We conclude that $\mathcal{E}(G,\pi)$ is a $\pi(G)'$-module.
\end{proof}

\begin{proof}[Proof of Theorem \ref{main finite}]
The theorem follows directly from Proposition \ref{module unique decomposition}, Proposition \ref{finite implies evanescent} and Proposition \ref{finite implies module}.
\end{proof}

%\begin{thm}
%Let $G$ be a locally compact group and suppose that $G$ admits a non-trivial irreducible representation $\pi$ that has almost invariant vectors. Then there exists an affine isometric action $\alpha$ of $G$ such that :
%\begin{enumerate}[\rm (i)]
%\item The linear part of $\alpha$ is a multiple of $\pi$.
%\item $\alpha$ is evanescent.
%\item $\alpha$ admits a decomposition into a product of an evanescent action and a nontrivial irreducible action.
%\item $\alpha$ admits a decomposition into a product of two irreducible  actions.
%\end{enumerate}
%\end{thm}
\begin{proof}[Proof of Theorem \ref{not unique decomposition}]
Depending on wether $\pi$ is real or complex, we fix $\K=\R$ or $\K=\C$. Let $H$ be the Hilbert space on which $\pi$ acts. Define the representation $\rho=\pi \otimes \id$ acting on $H \otimes \ell^2_{\K}(\Z)$. This will be the linear part of the action $\alpha$ that we will construct. Let $\mu$ be an adapted probability measure on $G$ and let $\Delta \in \pi(G)''$ be the corresponding laplacian operator, i.e.\ 
$$\Delta=\int_G | \id - \pi(g) |^2 \rd \mu(g) = 2 \int_G (\id-\pi(g)) \rd \mu(g). $$
Let $T=\Delta^{1/2}$. Since $\pi$ has almost invariant vectors but no invariant vectors, we know that $T$ is injective but contains $0$ in its spectrum. Thus, we can find an orthonormal set $(e_n)_{n \in \Z}$ in the domain of $T^{-1}$ such that 
$$ \sum_{n \in \Z} n^2\| Te_n\|^2 < +\infty.$$
 Now, set 
$$\xi=\sum_{n} nTe_n \otimes \delta_n \in H \otimes \ell^2_\K(\Z).$$
\begin{newclaim}
Take $S \in \B(\ell^2_\K(\Z))$. Then $(\id \otimes S)\xi$ is in the range of $T \otimes \id$ if and only if
$$ \sum_i i^2 \| S \delta_i\|^2 < +\infty.$$
\end{newclaim}
\begin{proof}[Proof of the claim]
Suppose that $(\id \otimes S)\xi=(T \otimes \id)\zeta$ for some $\zeta \in H \otimes \ell_\K^2(\Z)$. Then for all $i,j \in \Z$, we have the following computation :
\begin{align*}
\langle e_i \otimes \delta_j,\zeta \rangle & =   \langle (T^{-1}e_i) \otimes \delta_j, (T \otimes \id) \zeta \rangle \\ 
& =   \langle (T^{-1}e_i) \otimes \delta_j, (\id \otimes S) \xi \rangle \\ 
& = \sum_{n} \langle (T^{-1}e_i) \otimes \delta_j, nTe_n \otimes S\delta_n \rangle \\ 
& = \sum_{n} n \langle e_i , e_n \rangle  \langle \delta_j,   S\delta_n \rangle \\ 
& =i\langle \delta_j,S\delta_i \rangle.
\end{align*}
It follows that
$$ \sum_i i^2 \| S \delta_i\|^2=\sum_{i,j} i^2|\langle \delta_j,S\delta_i \rangle|^2 =\sum_{i,j} |\langle e_i \otimes \delta_j,\zeta \rangle|^2=  \| \zeta \|^2 < +\infty.$$
Conversely, if
$$ \sum_i i^2 \| S \delta_i\|^2 < +\infty$$
then we can define a vector $\zeta \in H \otimes \ell^2_\K(\Z)$ such that $(\id \otimes S)\xi=(T \otimes \id) \zeta$ by the formula
$$ \zeta=\sum_{i,j} i \langle \delta_j,S\delta_i\rangle (e_i \otimes \delta_j).$$
\end{proof}
Now, we can construct the desired affine isometric action. Let $\partial : H \otimes \ell^2_\K(\Z) \rightarrow Z^1_\mu(G,\rho)$ be the coboundary operator and let $U : H \otimes \ell^2_\K(\Z) \rightarrow Z^1_\mu(G,\rho)$ be the isometry obtained from its polar decomposition $\partial =U T$. Let $c=U\xi \in Z^1_\mu(G,\rho)$. Then $c \in \overline{B^1(G,\rho)}$ and thanks to the claim and Proposition \ref{polar decomposition} we know that for $S \in \B(\ell^2_\K(\Z))$, we have $(\id \otimes S)c \in B^1(G,\rho)$ if and only if 
$$ \sum_{n \in \Z} n^2 \| S \delta_n\|^2 < +\infty.$$
Let us now prove that the affine isometric action $\alpha$ associated to $\rho$ and $c \in Z^1(G,\rho)$ satisfies the properties $\rm (ii)$, $\rm (iii)$ and $\rm (iv)$.

Proof of $\rm (ii)$. For all $n \in \N$, let $P_n \in \B(\ell^2_\K(\Z))$ be the projection on the finite dimensional subspace $\Vect (\delta_0,\dots,\delta_n)$. Then $P_n$ is increasing, converges to $\id$ and $(\id \otimes P_n) c \in B^1(G,\rho)$ for all $n \in \N$, hence $\alpha$ is evanescent.

Proof of $\rm (iii)$. Define the vector $$v =\sum_{n > 0} \frac{1}{n} \delta_n \in \ell^2_\K(\Z).$$ Let $P$ be the rank-one projection on $v$. Consider the decomposition $\alpha = \beta_0 \oplus \beta_1$ where $\beta_0$ is the projected action associated to $\id \otimes P$ and $\beta_1$ is the projected action associated to $\id \otimes P^{\perp}$. Observe that $$\sum_{n \in \Z} n^2 \| P\delta_n\|^2=+\infty.$$ Thus $(\id \otimes P)c$ is not a coboundary, i.e.\ $\beta_0$ has no fixed point. Since the linear part of $\beta_0$ is $\pi$, which is irreducible, it follows that $\beta_0$ is an irreducible affine isometric action. Let us show that $\beta_1$ is evanescent. We can find an orthonormal basis $(b_n)_{n \in \Z}$ of $\{ f \in \ell^2_\K(\Z) \mid \langle f,v\rangle=0 \}$ such that for every $n \in \N$, the vector $b_n$ is finitely supported as a function on $\Z$. Let $Q_n$ be the projection on $\Vect (b_0,\dots, b_n)$ for all $n \in \N$. Then we have 
$$ \sum_{k \in \Z} k^2 \| Q_n \delta_k\|^2 < +\infty$$
because this sum is finite. This means that $(\id \otimes Q_n)c$ is a coboundary for all $n \in \N$. Since $\lim_n Q_n=P^{\perp}$, we conclude that $\beta_1$ is evanescent.

Proof of $\rm (iv)$. Let $\T=\{ z \in \C \mid |z|=1\}$. For every $f \in \ell^2_\K(\Z)$, let $\hat{f} \in \rL^2(\T)$ be its Fourrier transform. Then $f \mapsto \hat{f}$ is an isometry. If $\K=\C$, it is surjective, if $\K=\R$ its range is $\{ h \in \rL^2(\T) \mid \forall z \in \T, \; h(\bar{z})=\overline{h(z)} \}$. Let $\T=A_1 \sqcup A_2$ be a measurable partition of $\T$ such that for every $i \in \{1,2\}$ and every nonempty open subset $V \subset \T$, the set $A_i \cap V$ has positive measure. If $\K=\R$, assume also that $A_i$ is invariant under complex conjugation. Let $E_i=\{ f \in \ell^2_\K(\Z) \mid \hat{f} \text{ vanishes a.e.\ on } A_i \}$. Then $\ell^2_\K(\Z)=E_1 \oplus E_2$. Let $P_i$ be the orthogonal projection onto $E_i$. Let $\alpha_i$ be the projected action of $\alpha$ associated to $\id \otimes P_i$. Then $\alpha=\alpha_1 \oplus \alpha_2$ and we claim that $\alpha_i, i=1,2$ is irreducible. Indeed, suppose, without loss of generality, that $\alpha_1$ is not irreducible. Then there exists a nonzero orthogonal projection $Q \in \B(\ell^2_\K(\Z))$ with $Q \leq P_1$ such that $(\id \otimes Q)c$ is a coboundary, or equivalently, such that
$$ \sum_{n \in \Z} n^2 \| Q \delta_n \|^2 < +\infty $$
Let $f \in \ell^2_\K(\Z)$ be a function in the range of $Q$. Then we have 
$$ \sum_n n^2 |f(n)|^2 < +\infty.$$
This implies, in particular, that $\widehat{f}$ is a continuous function on $\T$. Since $Q \leq P_1$, we have $f \in E_1$, hence $\widehat{f}$ vanishes almost everywhere on $A_1$. Thanks to the continuity of $\widehat{f}$ and our assumption on $A_1$, this forces $\widehat{f}=0$, hence $f=0$. This contradicts the fact that $Q \neq 0$. We conclude that $\alpha_1$ is irreducible, as we wanted.
\end{proof}

\begin{proof}[Proof of Theorem \ref{equivalence WIE}]
We observe that $(\rm iii) \Rightarrow (\rm i)$ follows from Theorem \ref{not unique decomposition} while $(\rm ii) \Rightarrow (\rm iii)$ follows from Proposition \ref{module unique decomposition} because $(\rm ii)$ implies that $\mathcal{E}(G,\pi)=\overline{B^1(G,\pi)}$ is a $\pi(G)'$-module for every representation $\pi$. Therefore, we just have to prove that $\rm (i) \Rightarrow \rm (ii)$. 

Assume that $G$ is WIE. Let $\alpha$ be an affine isometric action of $G$ on $H$ that has almost fixed points. Let its linear part be $\pi$ and the associated cocycle $c \in Z^1(G,\pi)$. We will use the theory of direct integral decompositions of representations. For an introduction to this theory, we refer to \cite[Section 1.G]{BH19}. We can find a standard probability space $(X,\mu)$, a measurable field of Hilbert spaces $(H_x)_{x \in X}$ and a measurable field of irreducible representations $(\pi_x)_{x \in X}$ such that $H=\int_X^{\oplus} H_x \rd \mu(x)$ and $\pi=\int_X^{\oplus} \pi_x \rd \mu(x)$. We can also desintegrate $c \in Z^1(G,\pi)$ into a measurable field of cocycles $c_x \in Z^1(G,\pi_x)$. Since $c \in \overline{B^1(G,\pi)}$, we have $c_x \in \overline{B^1(G,\pi_x)}$ for a.e.\ $x \in X$. Since $\pi_x$ is irreducible and $G$ is WIE we know that $\pi_x$ has no almost invariant vectors, hence $B^1(G,\pi_x)$ is closed in $Z^1(G,\pi_x)$. Thus $c_x$ is a coboundary for a.e.\ $x \in X$. Write $c_x=\partial \xi_x$ for some measurable section $\xi_x \in H_x, \: x \in X$ (note that $\xi_x$ is uniquely determined, unless $\pi_x$ is trivial, in which case we take $\xi_x=0$). Observe that if $\int_X \| \xi_x\|^2 \rd \mu(x) < +\infty$ then the section $(\xi_x)_{x \in X}$ defines an element $\xi \in H$ such that $c=\partial \xi \in B^1(G,\pi)$. However, in general we will have $\int_X \| \xi_x\|^2 \rd \mu(x) = +\infty$. In that case, we can still find an increasing sequence of measurable subsets $A_n \subset X$ such that $\bigcup_n A_n=X$ and $\int_{A_n} \| \xi_x\|^2 \rd \mu(x) < +\infty$ for all $n \in \N$. Let $\xi_n=\int_{A_n}^{\oplus} \xi_x \rd \mu(x) \in H$. Let $P_n \in \pi(G)'$ be the projection on $\int_{A_n}^{\oplus} H_x \rd \mu(x)$. Then we have $P_n c=\partial \xi_n \in B^1(G,\pi)$ for all $n$. And since $(P_n)_{n \in \N}$ increases to $1$, we conclude that $c$ is evanescent as we wanted.
\end{proof}

\bibliographystyle{plain}

\end{document}